\documentclass[12pt]{amsart}
\usepackage{amsmath}
\usepackage[latin1]{inputenc}
\usepackage{amssymb, amsmath}

\theoremstyle{plain}
\newtheorem{teo}{Theorem}[section]
\newtheorem{lem}[teo]{Lemma}
\newtheorem{cor}[teo]{Corollary}
\newtheorem{prop}[teo]{Proposition}
\newtheorem*{clm}{Claim}

\theoremstyle{definition}
\newtheorem{fact}[teo]{Fact}
\newtheorem{remark}[teo]{Remark}

\newtheorem{defi}[teo]{Definition}
\newtheorem*{quest}{Question}
\def\beh{\begin{clm}}
\def\ebeh{\end{clm}}
\def\bewbeh{\par\noindent{\em Proof of Claim: }}

\def\Ind{\setbox0=\hbox{$x$}\kern\wd0\hbox to 0pt{\hss$\mid$\hss}
\lower.9\ht0\hbox to 0pt{\hss$\smile$\hss}\kern\wd0}
\def\Notind{\setbox0=\hbox{$x$}\kern\wd0\hbox to 0pt{\mathchardef
\nn=12854\hss$\nn$\kern1.4\wd0\hss}\hbox to
0pt{\hss$\mid$\hss}\lower.9\ht0 \hbox to
0pt{\hss$\smile$\hss}\kern\wd0}
\def\ind{\mathop{\mathpalette\Ind{}}}
\def\nind{\mathop{\mathpalette\Notind{}}}

\newcommand{\sub}{\subseteq}
\def\Aut{\mathrm{Aut}}
\def\bdd{\mathrm{bdd}}
\def\dcl{\mathrm{dcl}}
\def\acl{\mathrm{acl}}
\def\cl{\mathrm{cl}}
\def\cb{\mathrm{Cb}}
\def\tp{\mathrm{tp}}
\def\stp{\mathrm{stp}}
\def\Lstp{\mathrm{Lstp}}
\def\Q{\mathbb Q}
\def\P{\mathcal P}

\title[Elimination of Hyperimaginaries and Stable Independence]{Elimination of Hyperimaginaries and Stable Independence in simple CM-trivial theories}
\author{D. Palac\'\i n and F. O. Wagner}
\thanks{The first author was partially supported
by research project MTM 2008-01545 of the Spanish government and
research project 2009SGR 00187 of the Catalan government. The second author was partially
supported by the research projct ANR-09-BLAN-0047 of the Agence Nationale de la Recherche.\newline
This work was partially done while the first author was visiting the Institut
Camille Jordan of Lyon. The first author wishes to express his
gratitude to the members of Lyon logic group for their hospitality.
He also would like to thank Enrique Casanovas for his valuable
comments.}
\address{Universitat de Barcelona; Departament de L\`ogica, Hist\`oria
i Filosofia de la Ci\`encia, Montalegre 6, 08001 Barcelona, Spain}
\address{Universit\'e de Lyon; CNRS; Universit\'e Lyon 1; Institut
Camille Jordan UMR5208, 43 boulevard du 11 novembre 1918, F--69622
Villeurbanne Cedex, France}
\email{dpalacin@ub.edu}
\email{wagner@math.univ-lyon1.fr}

\begin{document}

\begin{abstract} In a simple CM-trivial theory every hyperimaginary is interbounded with a sequence of finitary hyperimaginaries. Moreover, such a theory eliminates hyperimaginaries whenever it eliminates finitary hyperimaginaries.
In a supersimple CM-trivial theory, the independence relation is stable.
\end{abstract}

\maketitle

\section{Introduction}

An important notion introduced by Shelah for a first-order theory is
that of an {\em imaginary} element: the class of a finite tuple by a
$\emptyset$-definable equivalence relation. The construction obtained
by adding all imaginary elements to a structure does not change its
basic model-theoretic properties, but introduces a convenient context
and language to talk about quotients (by definable equivalence
relations) and {\em canonical parameters} of definable sets. In the
context of a stable theory it also ensures the existence of {\em
canonical bases} for arbitrary complete types, generalizing the
notion of a field of definition of an algebraic variety.

The generalization of stability theory to the wider class of simple
theories necessitated the introduction of hyperimaginaries, classes
of countable tuples modulo $\emptyset$-type-definable equivalence
relations. Although the relevant model-theory for hyperimaginaries
has been reasonably well understood \cite{hart-kim-pillay}, they
cannot simply be added as extra sorts to the underlying structure,
since inequality of two hyperimaginaries amounts to
non-equivalence, and thus {\em a priori} is an open, but not a closed
condition. While hyperimaginary elements are needed for the general
theory, all known examples of a simple theory {\em eliminate} them in
the sense that they are interdefinable (or at least interbounded) with a sequence of ordinary
imaginaries; the latter condition is called {\em weak elimination}. It has thus been asked (and even been conjectured):
\begin{quest} Do all simple theories eliminate
hyperimaginaries?\end{quest}
The answer is positive for stable
theories \cite{poizat-pillay}, and for supersimple theories
\cite{BPW}. Among non-simple theories, the relation of being
infinitely close in a non-standard real-closed field gives rise to
non-eliminable hyperimaginaries; Casanovas and the second author have
constructed non-eliminable hyperimaginaries in a theory without the
strict order property \cite{Casa-Wag}.

A hyperimaginary is {\em finitary} if it is the class of a finite
tuple modulo a type-definable equivalence relation. Kim \cite{Kim}
has shown that small theories eliminate finitary hyperimaginaries,
and a result of Lascar and Pillay \cite{Lascar-Pillay} states that
bounded hyperimaginaries can be eliminated in favour of finitary
bounded ones. We shall show that in a CM-trivial simple theory all
hyperimaginaries are interbounded with sequences of finitary
hyperimaginaries. We shall deduce that in such a theory hyperimaginaries
can be eliminated in favour of finitary ones. In
particular, a small CM-trivial simple theory eliminates
hyperimaginaries. However, even the question whether all one-based simple theories eliminate hyperimaginaries is still open.

Elimination of hyperimaginaries is closely related to another question, the {\em stable forking conjecture}:
\begin{quest} In a simple theory, if $a\nind_BM$ for some model $M$ containing $B$, is there a stable formula in $\tp(a/M)$ which forks over $B$~?\end{quest}
If we do not require $M$ to be a model, nor to contain $B$, this is called {\em strong stable forking}.
Every known simple theory has stable forking; Kim \cite{Kim2} has shown that one-based simple theories with elimination of hyperimaginaries have stable forking. Kim and Pillay \cite{KP} have strengthened this to show that one-based simple theories with
weak elimination of imaginaries
hyperimaginaries have strong stable forking; on the other hand pseudofinite fields (which are supersimple of SU-rank $1$) do not. Conversely, stable forking implies
weak elimination of hyperimaginaries (Adler).

While we shall not attack the stable forking conjecture as such, we shall show in the last section that the independence relation $x\ind_{y_1}y_2$ is stable, meaning that it cannot order an infinite indiscernible sequence.

\section{Preliminaries}

As usual, we shall work in the monster model $\mathfrak{C}$ of a
complete first-order theory (with infinite models), and all sets of
parameters and all sequences of elements will live in
$\mathfrak{C}^{eq}$. Given any sequences $a,b$ and any set of
parameters $A$, we write $a\equiv_A b$ whenever $a$ and $b$ have the
same type over $A$. We shall write $a\equiv_A^s b$ if in addition $a$
and $b$ lie in the same class modulo all $A$-definable finite
equivalence relations (i.e.\ if $a$ and $b$ have the same {\em strong
type} over $A$), and $a\equiv_A^{Ls} b$ if they lie in the same class
modulo all $A$-invariant bounded equivalence relations (i.e.\ if $a$
and $b$ have the same {\em Lascar strong type} over $A$). Recall that
a theory is $G$-compact over a set $A$ iff $\equiv^{Ls}_A$ is
type-definable over $A$ (in which case it is the finest bounded
equivalence relation type-definable over $A$). A theory $T$ is
$G$-compact whenever it is $G$-compact over any $A$. In particular,
simple theories are $G$-compact \cite{Kim}.

\begin{defi} A hyperimaginary $h$ is
{\em finitary} if $h\in \dcl^{heq}(a)$ for
some finite tuple $a$ of imaginaries, and {\em quasi-finitary} if
$h\in\bdd(a)$ for some finite tuple $a$ of imaginaries.\end{defi}
\begin{defi} A hyperimaginary $h$ is {\em eliminable} if it is
interdefinable with
a sequence $e=(e_i:i\in I)$ of imaginaries, i.e.\ if there is such a
sequence $e$ with $\dcl^{heq}(e)=\dcl^{heq}(h)$. A theory $T$ eliminates (finitary/quasi-finitary)
hyperimaginaries if all (finitary/quasi-finitary) hyperimaginaries
are eliminable in all models of $T$.\end{defi}
\begin{remark}\cite[Corollary 1.5]{Lascar-Pillay} If $h\in\dcl^{heq}(a)$, then there is a type-definable equivalence relation $E$ on $\tp(a)$ such that $h$ and the class $a_E$ of $a$ modulo $E$ are interdefinable.\end{remark}
\begin{lem}\label{elimhyp6} Let $e$ be a finitary hyperimaginary. If $T$ eliminates finitary hyperimaginaries,
then $T(e)$ eliminates finitary hyperimaginaries.\end{lem}
\begin{proof} Let $a$ be a finite tuple with $e\in\dcl^{heq}(a)$, and
$h$ a finitary hyperimaginary over $e$. So there is a finite tuple
$b$ with $h\in\dcl^{heq}(eb)\subseteq\dcl^{heq}(ab)$. Then there is a
type-definable equivalence relation $E$ over $\emptyset$ such that
$e$ and $a_E$ are interdefinable,
and a type-definable equivalence relation $F_a$ over $a$
such that $h$ and $b_{F_a}$ are interdefinable.
Moreover, $F_a$ only depends on the $E$-class of
$a$, that is, if $a'Ea$, then $F_{a'}=F_a$.

Type-define an equivalence relation by
$$xy\bar Euv\quad\Leftrightarrow\quad xEu\land yF_xv.$$
It is easy to see that $h$ is interdefinable with $(ab)_{\bar E}$ over $e$.
Moreover, $(ab)_{\bar E}$ is clearly finitary, and hence eliminable
in $T$. So $h$ is eliminable in $T(e)$.\end{proof}

The following fact appears in \cite[Proof of Proposition
2.2]{Lascar-Pillay}, but was first stated as such in \cite[Lemma 2.17]{BPW}.
\begin{fact}\label{elimhyp1} Let $h$ be a
hyperimaginary and let $a$ be a sequence of imaginaries such that
$a\in\bdd(h)$ and $h\in\dcl^{heq}(a)$. Then, $h$ is
eliminable.\end{fact}
\begin{fact}\cite[Lemma 2.18]{BPW} Let $h,e$ be hyperimaginaries with
$h\in\bdd(e)$. Then the set of $e$-conjugates of $h$ is interdefinable with
a hyperimaginary $h'$.\end{fact}
\begin{fact}\label{bddfinitary}\cite[Theorem 4.15]{Lascar-Pillay} A
bounded hyperimaginary is interdefinable with a sequence of finitary
hyperimaginaries.\end{fact}
\begin{prop}\label{fin-qfin} If $T$ eliminates finitary hyperimaginaries, then $T$ eliminates
quasi-finitary hyperimaginaries.\end{prop}
\begin{proof} Let $h$ be a quasi-finitary hyperimaginary and let $a$ be a finite tuple of
imaginaries such that $h\in\bdd(a)$. Consider $a'\equiv_ha$ with
$\bdd(a)\cap\bdd(a')=\bdd(h)$. Let $h'$ be the hyperimaginary
corresponding to the set of $aa'$-conjugates of $h$. Then $h'$ is
$aa'$-invariant, and hence finitary. It is thus interdefinable with a
sequence $e$ of imaginaries.

On the other hand, $h\in\bdd(a)\cap\bdd(a')$, as are all its
$aa'$-conjugates. Thus $h'\in\bdd(a)\cap\bdd(a')=\bdd(h)$. Hence
$e\in\acl^{eq}(h)$ and $h\in\bdd(h')=\bdd(e)$. By Fact
\ref{bddfinitary}, there is a sequence $h''$ of finitary
hyperimaginaries interdefinable with $h$ over $e$. By Lemma \ref{elimhyp6}
and elimination of finitary hyperimaginaries we see that $h''$ is interdefinable over
$e$ with a sequence $e'$ of imaginaries. So $h\in\dcl^{heq}(ee')$ and
$e'\in\dcl^{eq}(eh)$. Moreover, $ee'\in\acl^{eq}(h)$ since
$e\in\acl^{eq}(h)$. Hence $h$ is eliminable by Fact
\ref{elimhyp1}.\end{proof}

The following remarks and lemmata will need $G$-compactness.
\begin{remark}\label{elimhyp2} Let $T$ be $G$-compact over a set $A$.
The following are equivalent: \begin{enumerate} \item $a\equiv^{Ls}_A
b$ iff $a\equiv^{s}_A b$ for all sequences $a,b$. \item
$\Aut(\mathfrak{C}/\bdd(A))=\Aut(\mathfrak{C}/\acl^{eq}(A))$. \item
$\bdd(A)=\dcl^{heq}(\acl^{eq}(A))$.
\end{enumerate}\end{remark}
\begin{proof} Easy exercise.\end{proof}
\begin{remark}\label{elimhyp3} Let $T$ be a $G$-compact theory and
assume further that $a\equiv^{Ls}_A b\Leftrightarrow a\equiv^{s}_A b$
for all sequences $a,b$ and for any set $A$. Let now $h$ be a
hyperimaginary and let $e$ be a sequence of imaginaries such that $h$
and $e$ are interbounded. Then $h$ is eliminable.\end{remark}
\begin{proof} It follows from Remark \ref{elimhyp2} that
$\bdd(e)=\dcl^{heq}(\acl^{eq}(e))$. Fix an enumeration $\bar{e}$ of
$\acl^{eq}(e)$ and observe that $h\in \dcl^{heq}(\bar{e})$ and
$\bar{e}\in \bdd(h)$. Then apply Fact \ref{elimhyp1} to eliminate
$h$.\end{proof}

It turns out for $G$-compact theories that elimination of hyperimaginaries can be decomposed as weak elimination of hyperimaginaries plus the equality between Lascar strong types and strong types over parameter sets.
\begin{fact}\label{elimhyp5}\cite[Proposition
18.27]{Enrique-prebook} Assume that $T$ is $G$-compact. Then $T$
eliminates all bounded hyperimaginaries iff
$a\equiv^{Ls}b\Leftrightarrow a\equiv^{s}b$ for all sequences $a$,
$b$.\end{fact}
\begin{proof} The proof in \cite{Enrique-prebook} is
nice and intuitive; however, we will give another one using Remark
\ref{elimhyp3}. If $T$ eliminates bounded hyperimaginaries, then
$\Aut(\mathfrak{C}/\bdd(\emptyset))=\Aut(\mathfrak{C}/\acl^{eq}(\emptyset))$.
By Remark \ref{elimhyp2} we get $\Lstp=\stp$. For the other
direction, let $e\in \bdd(\emptyset)$ and let $\bar{a}$ be an
enumeration of $\acl^{eq}(\emptyset)$. It is clear that $e$ and
$\bar{a}$ are interbounded. By Remark \ref{elimhyp3}, $e$ is
eliminable.\end{proof}
\begin{lem}\label{elimhyp4} Suppose $T$ is
$G$-compact and assume further that $T$ eliminates finitary
hyperimaginaries. Then $a\equiv^{Ls}_A b$ iff $a\equiv^{s}_A b$ for
all sequences $a,b$ and for any set $A$.\end{lem}
\begin{proof} Since
$T$ is $G$-compact, it is enough to check the condition for finite
$A$. But then $T(A)$ eliminates finitary hyperimaginaries by Remark
\ref{elimhyp6}, and hence all bounded hyperimaginaries by Fact
\ref{bddfinitary}. Now applying Fact \ref{elimhyp5} we obtain that
$a\equiv^{Ls}b$ iff $a\equiv^{s} b$ in $T(A)$.\end{proof}

\section{Elimination of hyperimaginaries in simple
theories}

In this section $T$ will be a simple theory. Recall that the {\em canonical base} of $a$ over $b$, denoted $\cb(a/b)$, is the smallest definably closed subset $C$ of $\bdd(b)$ such that $a\ind_Cb$ and $\tp(a/C)$ is Lascar strong.

\begin{lem}\label{elimhyp7} For any $a$ and any $h\in\bdd(c)$ we have $\cb(a/h)\subseteq\dcl(ac)\cap\bdd(h)$. Therefore, the canonical base of the type of an imaginary finite tuple over a quasi-finitary hyperimaginary is finitary. Furthermore, if $b\in\cb(a/c)$ then $\dcl(ab)\cap\bdd(b)\subseteq\cb(a/c)$. In particular, if $c\in\dcl(a)$ then $\cb(a/c)=\dcl(a)\cap\bdd(c)$. \end{lem}
\begin{proof} Since $h\in\bdd(c)$, equality of Lascar strong types over $c$ refines equality of Lascar strong types over $h$, and the class of $a$ modulo the former is clearly in $\dcl(ac)$. So the class of $a$ modulo the latter is in $\dcl(ac)$, and $\cb(a/h)\in\dcl(ac)\cap\bdd(h)$. As a consequence, if $a$ is a finite tuple and $h$ is a quasi-finitary hyperimaginary bounded over some finite tuple $c$, then $\cb(a/h)$ is definable over the finite tuple $ac$.

For the second assertion put $b'=\dcl(ab)\cap\bdd(b)$. Since $b'\in\dcl(ab)$ there is an equivalence relation $E$ on $\tp(a/b)$ type-definable over $b$ such that $b'$ is interdefinable over $b$ with $a_E$. As $b'\in\bdd(b)$ and $b\in\cb(a/c)$, the $E$-class of $a$ is bounded over $\cb(a/c)$; as $\tp(a/\cb(a/c))$ is Lascar-strong, $a_E\in\cb(a/c)$.

The ``in particular'' clause is essentially \cite[Remark 3.8]{BPW}:
If $c\in\dcl(a)$ then clearly $c\in\cb(a/c)$; the assertion follows.
\end{proof}

Recall the definition of CM-triviality.
\begin{defi}A simple theory
$T$ is {\em CM-trivial} if for every tuple $a$ and for any sets
$A\sub B$ with $\bdd(aA)\cap\bdd(B)=\bdd(A)$ we have
$\cb(a/A)\sub\bdd(\cb(a/B))$.\end{defi}

\begin{remark} As in \cite[Corollary 2.5]{pillay}, in the definition of CM-triviality we may take $A\sub B$ to be models of the ambient theory and $a$ to be a tuple from the home sort. Therefore, it makes no difference in the definition of CM-triviality whether we consider hyperimaginaries or just imaginaries.\end{remark}

Now we characterize canonical bases in simple CM-trivial theories in terms of finitary hyperimaginaries.

\begin{prop}\label{lem:cb-CM} Assume the theory is simple CM-trivial. If $a$ is a finite imaginary tuple, then
$$\bdd(\cb(a/B))=\bdd(\cb(a/b):b\in X),$$
where $X$ is the set of all finitary $b\in\bdd(\cb(a/B))$.\end{prop}
\begin{proof} Since $\cb(a/b)\subseteq\bdd(b)\subseteq\bdd(\cb(a/B))$ for $b\in X$, we have
$$\bdd(\cb(a/b):b\in X)\subseteq\bdd(\cb(a/B)).$$
For the reverse inclusion, for every $b\in X$ let $\hat b$ be a real tuple with $\cb(a/b)\in\dcl(\hat b)$; we choose them such that
$$(\hat b:b\in X)\ind_{(\cb(a/b):b\in X)} aB,$$
whence $(\hat b:b\in X)\ind_B a$.

Now, if $a\nind_{(\hat b:b\in X)} B$ then there is a finite tuple $b'\in B\cup\{\hat b:b\in X\}$ and a formula $\varphi(x,b')\in\tp(a/B,\hat b:b\in X)$ which divides over $(\hat b:b\in X)$. Put $\bar b=\bdd(ab')\cap\bdd(B,\hat b:b\in X)$. Then $\bar b$ is a quasi-finitary hyperimaginary, and by CM-triviality
$$\cb(a/\bar b)\subseteq\bdd(\cb(a/B,\hat b:b\in X))=\bdd(\cb(a/B)).$$
Since $\cb(a/\bar b)$ is finitary by Lemma \ref{elimhyp7}, it belongs to $X$. Note that $b'\in\bar b$; but $a\ind_{\cb(a/\bar b)} \bar b$, so $\varphi(x,b')$ cannot divide over $\cb(a/\bar b)$, and even less over $(\hat b:b\in X)$ as this contains $\widehat{\cb(a/\bar b)}$. Thus, $a\ind_{(\hat b:b\in X)} B$, whence $a\ind_{(\cb(a/b):b\in X)} B$ by transitivity. Therefore $$\cb(a/B)\subseteq\bdd(\cb(a/b):b\in X).$$
\end{proof}

\begin{quest} The same proof will work without assuming CM-triviality if for every finite tuple $b\in B$ there is some quasi-finitary hyperimaginary $\bar b\in\bdd(B)$ with $b\in\dcl(\bar b)$ such that $\cb(a/\bar b)\subseteq\bdd(\cb(a/B))$.
Is this true in general?\end{quest}

We can now state (and prove) the main result.
\begin{teo}\label{mainteo} Let $T$ be a simple CM-trivial theory. Then every hyperimaginary is interbounded with a sequence of finitary hyperimaginaries.\end{teo}
\begin{proof} By Lemma \ref{elimhyp7} every hyperimaginary is interbounded with a canonical base. Since $\cb(A/B)$
is interdefinable with $\bigcup\{\cb(\bar a/B):\bar a\in A\mbox{ finite}\}$, it is enough to show that canonical bases of types of finite tuples are interbounded with sequences of finitary hyperimaginaries. This is precisely Proposition \ref{lem:cb-CM}.\end{proof}

\begin{cor}\label{cor:main} A simple CM-trivial theory eliminates hyperimaginaries whenever it eliminates finitary ones.\end{cor}
\begin{proof} By Theorem \ref{mainteo} every hyperimaginary is interbounded with a sequence of finitary hyperimaginaries and so with a sequence of imaginaries.  Since $T$ is
simple, it is $G$-compact, whence $\Lstp=\stp$ over any set by Lemma
\ref{elimhyp4}. We conclude that every hyperimaginary is eliminable by Remark
\ref{elimhyp3}.\end{proof}

\begin{cor} Every small simple CM-trivial theory eliminates
hyperimaginaries.\end{cor}
\begin{proof} A small simple theory
eliminates finitary hyperimaginaries by \cite{Kim}. Now apply Corollary
\ref{cor:main}.\end{proof}

\section{Stable independence for CM-trivial theories}

Recall that an $\emptyset$-invariant relation $R(x,y)$ is {\em stable} if there is no infinite indiscernible sequence $(a_i,b_i:i<\omega)$ such that $R(a_i,b_j)$ holds if and only if $i<j$. In this section, we shall show that independence is a stable relation, even with varying base set. We hope that this will help elucidate the stable forking problem.

\begin{teo}\label{stabind} In a supersimple CM-trivial theory, the relation $R(x;y_1y_2)$ given by $x\ind_{y_1}y_2$ is stable.\end{teo}
\begin{proof} Suppose not. Then there is an indiscernible sequence $I=(a_i:i\in\Q)$ and tuples $b,c$ such that\begin{itemize}                                                                                                          \item $I^+=(a_i:i>0)$ is indiscernible over $I^-bc$,
\item $I^-=(a_i:i<0)$ is indiscernible over $I^+bc$, and
\item $a_i\ind_cb$ if and only if $i>0$.\end{itemize}
We consider limit types with respect to the cut at $0$. Put 
$$p=\lim(I/I),\quad p^+=\lim(I^+/Ibc)\quad\text{and}\quad p^-=\lim(I^-/Ibc).$$
By finite satisfiability, $p^+$ and $p^-$ are both non-forking extensions of $p$, which is Lascar-strong. Let
$$A=\cb(p)=\cb(p^+)=\cb(p^-)\in\bdd(I^+)\cap\bdd(I^-).$$
As $p^+$ and $p^-$ do not fork over $A$, we have
$$a_i\ind_AI^+bc\text{ for all }i<0,\quad\text{and}\quad a_i\ind_AI^-bc\text{ for all }i>0.$$
We consider first $e_0=\bdd(a_1c)\cap\bdd(Ac)$. Then
$$\bdd(a_1e_0)\cap\bdd(Ae_0)=e_0.$$
Put $A_0=\cb(a_1/e_0)$. By CM-triviality
$$A_0\in e_0\cap\bdd(\cb(a_1/Ae_0))\subseteq\bdd(a_1c)\cap\bdd(A),$$
since $a_1\ind_Ae_0$ implies $\cb(a_1/Ae_0)\subseteq\bdd(A)$.

Note that $a_1\ind_cb$ yields $a_1\ind_{A_0c}b$. Moreover $c\in e_0$, so $a_1\ind_{A_0}e_0$ implies $a_1\ind_{A_0}c$, whence $a_1\ind_{A_0}cb$ by transitivity. On the other hand, suppose $bc\ind_{A_0}a_{-1}$. Then $b\ind_{A_0c}a_{-1}$; as $b\ind_ca_1$ implies $b\ind_cA_0$, we obtain $b\ind_c A_0a_{-1}$, contradicting $a_{-1}\nind_cb$. Therefore $bc\nind_{A_0}a_{-1}$. Since $I$ remains indiscernible over $A_0,$ and both $I^+$ and $I^-$ remain indiscernible over $A_0bc$, we may add $A_0$ to the parameters and suppose $c=\emptyset$ (replacing $b$ by $bc$).
\begin{fact}\label{decomp}\cite[Theorem 5.2.18]{Wagner-book} In a supersimple theory, for any finitary $a$ there are some $B\ind a$ and a hyperimaginary finite tuple $\bar a$ of independent realizations of regular types over $B$, such that $\bar a$ is domination-equivalent with $a$ over $B$.\end{fact}
By Fact \ref{decomp} there are $B\ind a_1$ and an independent tuple $\bar a_1$ of realizations of regular types over $B$ such that $\bar a_1$ is domination-equivalent with $a_1$ over $B$. Since $B\ind a_1$ and $I$ is indiscernible, we may assume by \cite[Theorem 2.5.4]{Wagner-book} that
$Ba_i\equiv Ba_1$ for all $i\in\Q$, and $B\ind I$. So there are $\bar a_i$ for $i\in\Q$ with $Ba_i\bar a_i\equiv Ba_1\bar a_1$. We can also assume $B\ind_Ib$, whence $B\ind Ib$. In particular $b\ind_{a_i}B$, so for $i>0$ we obtain $b\ind Ba_i$ and thus $b\ind_Ba_i$, while for $i<0$ we have $b\nind a_iB$ and $b\ind B$, whence $b\nind_Ba_i$. By domination-equivalence, $\bar a_i\ind_B b$ for $i>0$ whereas $\bar a_i\nind_B b$ for $i<0$.

By compactness and Ramsey we may suppose in addition that $\bar I=(\bar a_i:i\in\Q)$ is $B$-indiscernible, $\bar I^+=(\bar a_i:i>0)$ is indiscernible over $Bb\bar I^-$ and $\bar I^-=(\bar a_i:i<0)$ is indiscernible over $Bb\bar I^+$. We shall add $B$ to the parameters and suppress it from the notation.
We may further assume that $\bar a_{-1}'\ind b$ for any proper subtuple $\bar a_{-1}'\subseteq\bar a_{-1}$.
\beh All the regular types in $\bar a_i$ are non-orthogonal.\ebeh
\bewbeh Consider $c, c'\in\bar a_{-1}$ and put $\bar c=\bar a_{-1}\setminus\{c,c'\}$. Then $\bar cc\ind b$ and $\bar cc'\ind b$ by minimality, whence $c\ind b\bar c$ and $c'\ind b\bar c$, as $\bar a_{-1}=\bar ccc'$ is an independent tuple.

Suppose $c\ind_{b\bar c}c'$. Then $c\ind b\bar cc'$, whence $c\ind_{\bar cc'}b$ and finally $b\ind \bar ccc'$, contradicting $b\nind \bar a_{-1}.$ So $\tp(c/b\bar c)$ and $\tp(c'/b\bar c)$ are non-orthogonal; as they do not fork over $\emptyset$ we get $\tp(c)$ non-orthogonal to $\tp(c')$. The claim now follows, as all $\bar a_i$ have the same type over $\emptyset$.\qed

Let $w_\P(.)$ denote the weight with respect to that non-orthogonality class $\P$ of regular types. Then $\bar a_i$ is $\P$-semi-regular; since $\bar a_{-1}\nind b$ we obtain 
$$w_\P(\bar a_{-1})>w_\P(\bar a_{-1}/b)$$
by \cite[Lemma 7.1.14]{Pi96} (the proof works just as well for the simple case).

We again consider limit types with respect to the cut at $0$. Put
$$\bar p=\lim(\bar I/\bar I),\quad\bar p^+=\lim(\bar I^+/\bar Ib)\quad\text{and}\quad\bar p^-=\lim(\bar I^-/\bar Ib).$$
Once more, $\bar p$ is Lascar-strong and $\bar p^+$ and $\bar p^-$ are non-forking extensions of $\bar p$ by finite satisfiability; let
$$\bar A=\cb(\bar p)=\cb(\bar p^+)=\cb(\bar p^-)\in\bdd(\bar I^+)\cap\bdd(\bar I^-).$$
As before,
$$\bar a_i\ind_{\bar A}\bar I^+b\text{ for all }i<0,\quad\text{and}
\quad\bar a_i\ind_{\bar A}\bar I^-b\text{ for all }i>0.$$
Put $e_1=\bdd(\bar a_{-1}b)\cap\bdd(\bar Ab)$. Then
$$\bdd(\bar a_{-1}e_1)\cap\bdd(\bar Ae_1)=e_1.$$
Let $A_1=\cb(\bar a_{-1}/e_1)$. By CM-triviality
$$A_1\in e_1\cap\bdd(\cb(\bar a_{-1}/\bar Ae_1))\subseteq\bdd(\bar a_{-1}b)\cap\bdd(\bar A),$$
since $\bar a_{-1}\ind_{\bar A}e_1$ implies $\cb(a_{-1}/\bar Ae_1)\subseteq\bdd(\bar A)$.

As $b\in e_1$ and $\bar a_{-1}\ind_{A_1}e_1$ we obtain $\bar a_{-1}\ind_{A_1}b$. Moreover $\bar a_1\equiv_{A_1}\bar a_{-1}$, since $A_1\subseteq\bdd(\bar A)$ and $\bar I$ remains indiscernible over $\bar A$. Therefore
$$w_\P(\bar a_{-1}/A_1b)=w_\P(\bar a_{-1}/A_1)=w_\P(\bar a_1/A_1).$$
Recall that $A_1\subseteq\bdd(\bar a_{-1}b)$.
Then $$\begin{aligned}w_\P(\bar a_{-1}/b)&=w_\P(\bar a_{-1}A_1/b)\\
&=w_\P(\bar a_{-1}/A_1b)+w_\P(A_1/b)\\
   &=w_\P(\bar a_1/A_1)+w_\P(A_1/b)\\
   &\ge w_\P(\bar a_1/A_1b)+w_\P(A_1/b)\\
   &=w_\P(\bar a_1A_1/b)\ge w_\P(\bar a_1/b)\\
&=w_\P(a_1)=w_\P(a_{-1})>w_\P(\bar a_{-1}/b).\end{aligned}$$
This final contradiction proves the theorem.\end{proof}

\begin{remark} Note that the proof only uses the conclusion of Fact \ref{decomp}. The theorem thus still holds for simple CM-trivial theories with finite weights (strongly simple theories) and enough regular types, for instance CM-trivial simple theories without dense forking chains.\end{remark}

\begin{quest} By \cite[Theorem 4.20]{PaWa} it is sufficient to assume that every regular type is CM-trivial, as this implies global CM-triviality. However, for a regular type $p$ a more general notion of CM-triviality is often more appropriate, namely
$$\cl_p(aA)\cap\cl_p(B)=\cl_p(A)\quad\Rightarrow\quad\cb(a/\cl_p(A))\subseteq\cl_p(\cb(a/\cl_p(B))).$$
If this holds for all regular types $p$, is independence still stable?\end{quest}

\begin{cor} An $\omega$-categorical supersimple CM-trivial theory has stable forking.\end{cor}
\begin{proof} Suppose $A\nind_B C$. Then there are finite tuples $\bar a\in A$ and $\bar c\in C$ with $\bar a\nind_B\bar c$. By supersimplicity, there is a finite $\bar b\in B$ with $\bar a\bar c\ind_{\bar b}B$. Thus $\bar a\nind_{\bar b}\bar c$. By $\omega$-categoricity there is a formula $\varphi(\bar x,\bar y_1\bar y_2)$ which holds if and only if $\bar x\nind_{\bar y_1}\bar y_2$. Then $\varphi$ is stable by Theorem \ref{stabind}, and $\varphi(\bar x,\bar b\bar c)\in\tp(\bar a/\bar b\bar c)$.\end{proof}

Let $\Sigma$ be an $\emptyset$-invariant family of types. Recall the definition of $\Sigma$-closure:
$$\cl_\Sigma(A)=\{a:\tp(a/A)\text{ is $\Sigma$-analysable}\}.$$
\begin{fact}\label{clind}\cite[Lemma 3.5.3 and 3.5.5]{Wagner-book} If $\dcl(AB)\cap\cl_\Sigma(A)\subseteq\bdd(A)$, then $B\ind_A\cl_\Sigma(A)$. If $A\ind_BC$, then $A\ind_{\cl_\Sigma(B)}C$.\end{fact}

\begin{cor}\label{stabindcl} In a supersimple CM-trivial theory the relation $R(x;y_1y_1)$ given by $x\ind_{\cl_\Sigma(y_1)} y_2$ is stable.\end{cor}
\begin{proof} Suppose not. Then there is an indiscernible sequence $I=(a_i:i\in\Q)$ and tuples $b$, $c$ such that\begin{itemize}                                                                                                          \item $I^+=(a_i:i>0)$ is indiscernible over $I^-bc$,
\item $I^-=(a_i:i<0)$ is indiscernible over $I^+bc$, and
\item $a_i\ind_{\cl_\Sigma(c)}b$ if and only if $i>0$.\end{itemize}
Put $c'=\dcl(bc)\cap\cl_\Sigma(c)$. By Fact \ref{clind} we have $b\ind_{c'}\cl_\Sigma(c)$, so by transitivity $a_i\ind_{c'}b$ for $i>0$. Suppose $a_i\ind_{c'}b$ for $i<0$. Since $\cl_\Sigma(c')=\cl_\Sigma(c)$, Fact \ref{clind} yields $a_i\ind_{\cl_\Sigma(c)}b$, a contradiction. Thus $a_i\ind_{c'}b$ if and only if $i>0$, contradicting Theorem \ref{stabind}.\end{proof}

To conclude the paper we prove a version of Corollary \ref{stabindcl} without the assumption of CM-triviality, but for a particular $\emptyset$-invariant family, namely the family $\P$ of all non one-based types.
\begin{fact}\label{fact:sigmabase}\cite[Corollary 5.2]{PaWa} In a simple theory $a\ind_{\cl_\P(a)\cap\bdd(b)} b$ for all tuples $a$ and $b$, where $\P$ is the family of all non one-based types.\end{fact}

\begin{teo}\label{stabindP} In a simple theory, the relation $R(x;y_1y_2)$ given by $x\ind_{\cl_\P(y_1)} y_2$ is stable, where $\P$ is the family of all non one-based types.\end{teo}
\begin{proof} Suppose not. Then there is an indiscernible sequence $I=(a_i:i\in\Q)$ and tuples $b$, $c$ such that\begin{itemize}                                                                                                          \item $I^+=(a_i:i>0)$ is indiscernible over $I^-bc$,
\item $I^-=(a_i:i<0)$ is indiscernible over $I^+bc$, and
\item $a_i\ind_{\cl_\P(c)} b$ if and only if $i>0$.\end{itemize}
As before, we consider limit types with respect to the cut at $0$. Let
$$p=\lim(I/I),\quad p^+=\lim(I^+/Ib)\quad\text{and}\quad p^-=\lim(I^-/Ib).$$
By finite satisfiability, $p^+$ and $p^-$ are both non-forking extensions of $p$, which is Lascar-strong. Let
$$A=\cb(p)=\cb(p^+)=\cb(p^-)\in\bdd(I^+)\cap\bdd(I^-).$$
As in the proof of Theorem \ref{stabind} we have
$$a_i\ind_AI^+bc\text{ for all }i<0,\quad\text{and}\quad a_i\ind_AI^-bc\text{ for all }i>0.$$
We consider first $e=\cl_\P(a_1)\cap\bdd(A)$. Then $a_1\ind_{e} A$ by Fact \ref{fact:sigmabase}; since $I$ remains indiscernible over $\bdd(A)$ we have $a_{-1}\equiv_{\bdd(A)} a_1$, whence $e=\cl_\P(a_{-1})\cap\bdd(A)$ and $a_{-1}\ind_e A$. On the other hand, since $e\in\bdd(A)$ and $a_i\ind_A bc$ for $i\in\Q$ we obtain
$$a_1\ind_e bc\quad\mbox{and}\quad a_{-1}\ind_e bc.$$
Now put $c'=\dcl(bc)\cap\cl_\P(c)$; note that $\cl_\P(c')=\cl_\P(c)$. Then $b\ind_{c'} \cl_\P(c)$ by Fact \ref{clind}. Moreover, $a_1\ind_{\cl_\P(c)} b$ yields $\cl_\P(a_1)\ind_{\cl_\P(c)} b$ by Fact \ref{clind}, whence $\cl_\P(a_1)\ind_{c'} b$. Thus $e\ind_{c'} b$ and hence $e\ind_{c'} bc$ since $c\subseteq c'$. But now $c'\subseteq\dcl(bc)$ and $a_{-1}\ind_e bc$ imply that $a_{-1}\ind_{c'} bc$. Hence $a_{-1}\ind_{\cl_\P(c)} b$ by Fact \ref{clind}, as $\cl_\P(c')=\cl_\P(c)$. This contradiction finishes the proof.\end{proof}

\begin{remark} If the theory is supersimple, we can take $\P$ to be the family of all non one-based {\em regular} types.\end{remark}
\begin{remark} Theorem \ref{stabindP} generalises the fact that independence is stable in a one-based theory. For a true generalisation of Theorem \ref{stabind} to arbitrary theories, one should take $\P$ to be the family of all $2$-ample types. This is work in progress.\end{remark}

\end{document}